\documentclass[11 pt]{article}

\usepackage{amsmath,amsthm,amsfonts,amssymb,csquotes,epsfig, titlesec, epsfig, float, caption,wrapfig, mathrsfs}
\usepackage[hidelinks]{hyperref}
\usepackage{tikz-cd}
\usepackage{tikz}
\usepackage{tikz-3dplot}
\usepackage{xcolor}
\usepackage{verbatim}
\usepackage{accents}
\usepackage{marvosym}
\usepackage{multirow}
\usepackage{multicol}
\usepackage{listings}
\usepackage[citestyle=alphabetic,bibstyle=alphabetic,backend=bibtex, maxbibnames=10,minbibnames=5]{biblatex}
\usepackage[colorinlistoftodos]{todonotes}
\usepackage[]{enumerate}
\usepackage[american]{babel}

\usetikzlibrary{calc, decorations.pathreplacing,shapes.misc}
\usetikzlibrary{decorations.pathmorphing}
\usepackage[left=1in,top=1in,right=1in]{geometry}
\usepackage[capitalize]{cleveref}
\usepackage{subcaption}
\usetikzlibrary{shapes.geometric}

\newtheorem{thm}{Theorem}[section]

\newtheorem{lemma}[thm]{Lemma}

\theoremstyle{remark} 
\newtheorem{rem}[thm]{Remark}

 \crefname{rem}{Remark}{Remarks}
 \Crefname{rem}{Remark}{Remarks}

\newtheorem{example}[thm]{Example}

\theoremstyle{definition} 
 
\newtheorem{definition}[thm]{Definition} 

\titleformat*{\section}{\normalsize \bfseries \filcenter}
\titleformat*{\subsection}{\normalsize \bfseries }
\Crefname{mainthm}{Theorem}{Theorems}

\Crefname{maincor}{Corollary}{Corollaries}

\Crefname{mainquestion}{Question}{Questions}

\Crefname{mainconj}{Conjecture}{Conjectures}

\crefformat{equation}{(#2#1#3)}

\makeatletter
\def\namedlabel#1#2{\begingroup
   \def\@currentlabel{#2}
   \label{#1}\endgroup
}
\makeatother

\DeclareFontFamily{U}{mathx}{}
\DeclareFontShape{U}{mathx}{m}{n}{<-> mathx10}{}
\DeclareSymbolFont{mathx}{U}{mathx}{m}{n}
\DeclareMathAccent{\widehat}{0}{mathx}{"70}
\DeclareMathAccent{\widecheck}{0}{mathx}{"71}

\renewbibmacro{in:}{}

\usetikzlibrary{decorations.markings}
\usepackage{amsmath}

\title{\normalsize \textbf{A short computation of the Rouquier dimension for a cycle of projective lines}}
\author{\normalsize Andrew Hanlon, Jeff Hicks}
\date{}

\newcommand{\RR}{\mathbb R}

\renewcommand{\AA}{\mathbb A}
\newcommand{\NN}{\mathbb N}
\newcommand{\PP}{\mathbb P}

\DeclareMathOperator{\ads}{ads}

\newcommand{\tensor}{\otimes}

\newcommand{\core}{\mathfrak{c}}

\newcommand{\gentime}{\text{\ClockLogo}}

\DeclareMathOperator{\Ddim}{Ddim}

\DeclareMathOperator{\Perf}{Perf}

\DeclareMathOperator{\id}{id}

\DeclareMathOperator{\cone}{cone}

\DeclareMathOperator{\Hom}{Hom}

\DeclareMathOperator{\Coh}{Coh}

\DeclareMathOperator{\st}{\; |\; }
\DeclareMathOperator{\Fuk}{Fuk}

\DeclareMathOperator{\Spec}{Spec}

\DeclareMathOperator{\Rdim}{Rdim}

\begin{document}

\maketitle
\begin{abstract}
	Given a dg category $\mathcal C$, we introduce a new class of objects (weakly product bimodules) in $\mathcal C^{op}\otimes \mathcal C$ generalizing product bimodules. We show that the minimal generation time of the diagonal by weakly product bimodules provides an upper bound for the Rouquier dimension of $\mathcal C$. As an application, we give a purely algebro-geometric proof of a result of \citeauthor{burban2017derived} that the Rouquier dimension of the derived category of coherent sheaves on an $n$-cycle of projective lines is one. Our approach explicitly gives the generator realizing the minimal generation time.
\end{abstract}

\section{Introduction}
The Rouquier dimension of a triangulated category $\mathcal C$ is the minimal generation time of the category by any object (see \cref{def:rdim} for a precise statement). One setting that has attracted significant interest comes from algebraic geometry, where Orlov conjectured that $\Rdim(D^b\Coh(X))= \dim(X)$ for any smooth quasi-projective variety $X$ \cite{orlov2009remarks}. Orlov's conjecture is known for a handful of examples and special classes of smooth varieties. When $X$ is a singular variety, there are known bounds on the Rouquier dimension in some circumstances. 
Both \cite{bai2023rouquier} and \cite{dey2023strong} provide summaries of known cases and bounds.
In this paper, we study $D^b\Coh(I_n)$ where $I_n$ is an $n$-cycle of projective lines.
In previous work, the Rouquier dimension of $D^b\Coh(I_n)$ was computed via comparison to triangulated categories that are not immediately related to algebraic geometry:
\begin{itemize}
	\item \citeauthor{burban2017derived} computed that the Rouquier dimension of the bounded derived category of a finite-dimensional module over a gentle or skew-gentle algebra is at most one \cite{burban2017derived}. Since these categories are categorical resolutions of the bounded derived category of coherent sheaves on a cycle of projective lines, \Citeauthor{burban2017derived} obtain as a corollary that $\Rdim(D^b\Coh(I_n))=1$. They also exhibited a minimal time generator for the example of a nodal Weierstrass cubic.
	\item Let $\check X$ be a real 2-dimensional Liouville domain. \Citeauthor{bai2023rouquier} showed that the Rouquier dimension of the wrapped Fukaya category $\mathcal{W}(\check X)$ of $\check X$ is at most one \cite{bai2023rouquier}.  By applying their work to $\check X$, the homological mirror of $I_n$ (satisfying  $\mathcal{W}(I_n)$ and $\Fuk(\check X)$ are derived equivalent \cite{lekili2020homological}) we immediately obtain that $\Rdim(D^b\Coh(I_n))=1$.
\end{itemize}

In this paper, we prove that $\Rdim(D^b\Coh(I_n))=1$ directly using algebro-geometric techniques. We provide an object that generates the diagonal $\mathcal O_\Delta\in D^b\Coh(I_n\times I_n)$ in time one and show that this generation time bounds the Rouquier dimension. Our approach explicitly identifies a generator $G$ satisfying $\gentime_G(D^b\Coh(I_n))=1$.

\subsection*{Outline}
In \cref{sec:symplectic}, we informally summarize our perspective of how minimal length cellular resolutions in algebraic geometry are related to Morse theory and symplectic geometry through homological mirror symmetry.
Outside of that section, the paper requires no background in symplectic geometry, and no other section is dependent on \cref{sec:symplectic}.
In \cref{sec:weakDiagonalDimension}, we recall the notions of Rouquier and diagonal dimension of a dg (or $A_\infty$) category $\mathcal C$, and we introduce a new class of objects (the weakly product bimodules, \cref{def:weaklyProduct}) in $\mathcal C^{op}\tensor \mathcal C$.
The main observation in this paper is that the Rouquier dimension is bounded by the minimal generation time of the diagonal by weakly product bimodules as shown in \cref{thm:weakProductToRdim}. 
We show in \cref{lem:makingWeaklyProduct} that the mapping cone of a product morphism between product bimodules in $\mathcal C^{op}\otimes \mathcal C$ is a weakly product bimodule.

With these algebraic tools in hand, we conclude the paper with \cref{sec:examples}, where we explicitly generate the diagonal in minimal time by weakly product bimodules for the toy cases of the affine line, the nodal conic (\cref{lem:resolutionNodalConic}), and the cycle of projective lines $I_n$ (\cref{thm:resIn}).

\subsection*{Acknowledgements}
The authors thank Alexey Elagin, Daniel Erman, Pat Lank, and Oleg Lazarev for insightful conversations. We are also grateful to the anonymous referee for helpful comments that improved the paper. The second author is supported by EPSRC Grant EP/V049097/1 (Lagrangians from Algebra and Combinatorics). This project arose from a problem session from the workshop ``Syzygies and Mirror Symmetry'' hosted by the American Institute of Mathematics. Many of the examples in this paper were initially worked out using Macaulay2 \cite{M2}.

\section{Resolutions of the diagonal from the perspective of symplectic geometry}
\label{sec:symplectic}
The \emph{diagonal dimension} of a dg category $\mathcal C$ is the minimal generation time of the diagonal bimodule via product objects (see \cref{def:ddim} for the precise definition).
The relation between resolving the diagonal and generation of $\mathcal C$ goes back to Beilinson and Kontsevich.
This can be refined to an upper bound on the Rouquier dimension by the diagonal dimension (see, for instance, \cite{rouquier2008dimensions,bondal2003generators};  we use conventions matching \cite{bai2023rouquier,hanlon2023relating}). In addition, one can computationally check the generation time of the diagonal by exhibiting a twisted complex that is quasi-isomorphic to the diagonal.
\begin{figure}[h]
	\centering
	\begin{subfigure}[t]{.4\linewidth}
		\centering
		\begin{tikzpicture}

\begin{scope}[decoration={
    markings,
    mark=at position 0.5 with {\arrow{>}}}
    ] 

\fill[gray!20]  (-1.5,-3) rectangle (-4.5,1);

\draw[dashed, postaction={decorate}] (-4.5,1) -- (-1.5,1);
\draw[dashed, postaction={decorate}](-4.5,-3) -- (-1.5,-3);

\draw[dotted] (-4.5,1) -- (-4.5,-3);

\draw (-1.5,-1.5) -- (-1.5,1) (-1.5,-3) -- (-1.5,-2.5);
\clip  (-1.5,-3) rectangle (-4.5,1);

\draw[dotted, fill=white]  (-1.5,-2) ellipse (0.5 and 0.5);

\draw[blue] (-4,1) .. controls (-4,0.5) and (-4,0.5) .. (-3.5,0) .. controls (-3,-0.5) and (-3,-0.5) .. (-3,-1) .. controls (-3,-1.5) and (-3,-1.5) .. (-3.5,-2) .. controls (-4,-2.5) and (-4,-2.5) .. (-4,-3);
\draw[blue] (-3,-1) .. controls (-2.5,-1) and (-2.5,-2) .. (-2,-2);

\end{scope}
\draw[red] (-3.5,1) -- (-3.5,-3);
\draw[red] (-3.5,0.5) .. controls (-3,0.5) and (-2,-0.5) .. (-1.5,-0.5);
\node[circle, fill, scale=.2] at (-3.5,0) {};
\node[circle, fill, scale=.2] at (-3.5,-2) {};
\node[left] at (-3.5,-2) {$1$};
\node[left] at (-3.5,0) {$0$};
\end{tikzpicture} 		\caption{Perturbation of the skeleton of $T^*S^1$ with one stop. The labels are the values of $H$.}
		\label{fig:skelA1a}
	\end{subfigure}
	\hspace{.1\linewidth}
	\begin{subfigure}[t]{.4\linewidth}
		\centering
		\begin{tikzpicture}

\begin{scope}[decoration={
    markings,
    mark=at position 0.5 with {\arrow{>}}}
    ] 

\fill[gray!20]  (-1.5,-3) rectangle (-4.5,1);

\draw[dashed, postaction={decorate}] (-4.5,1) -- (-1.5,1);
\draw[dashed, postaction={decorate}](-4.5,-3) -- (-1.5,-3);

\draw[dotted] (-4.5,1) -- (-4.5,-3);

\draw (-1.5,-1.5) -- (-1.5,1) (-1.5,-3) -- (-1.5,-2.5);
\clip  (-1.5,-3) rectangle (-4.5,1);
\def\int{1}

\node (v5) at (-4.5,-1.5) {};
\node (v2) at (-3.5,0) {$\mathcal O_{\mathbb A^1\times \mathbb A^1}$};
\node (v3) at (-3.5,-2) {$\mathcal O_{\mathbb A^1\times \mathbb A^1}$};

\draw (v3) edge[->]  node[fill=gray!20]{ $x_1$} (v2);
\draw   (-3.5,-3) edge node[fill=gray!20]{$-x_2$} (v3);

\draw  (v2) edge[<-] (-3.5,1);

\draw[dotted, fill=white]  (-1.5,-2) ellipse (0.5 and 0.5);

\end{scope}
\end{tikzpicture} 		\caption{A resolution of the diagonal in $\AA^1\times \AA^1$ by product objects}
		\label{fig:resA1a}
	\end{subfigure}
	\caption{To translate from symplectic to algebraic geometry: $T^*S^1$ with one stop is homologically mirror to $\AA^1$. At each intersection point $q\in \phi_H(\core_X)\cap \core_X$, we take the product of the sheaves mirror to the cocores of $q\in \core_X$ and $q\in \phi_H(\core_X)$.  We then order the terms by the value of the Hamiltonian.}
	\label{fig:A1a}
\end{figure}
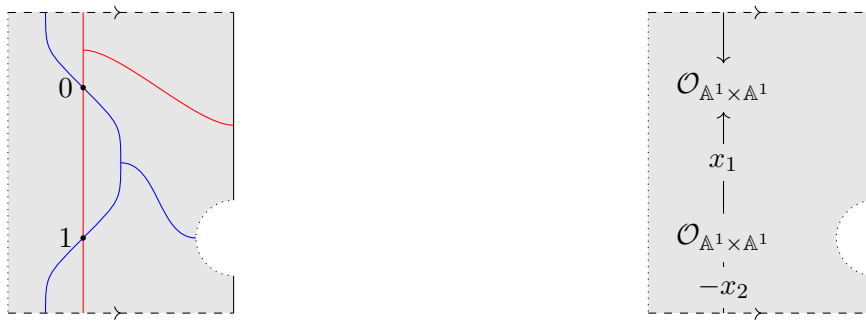
For example, the diagonal in $\AA^1\times \AA^1$ is easily seen to have a length one resolution by product objects :
\begin{equation}
	\mathcal O_{\AA^1\times \AA^1}\xleftarrow{x_1-x_2} \mathcal O_{\AA^1\times \AA^1}.
	\label{eq:resDiagA1}
\end{equation}

There are two difficulties with taking this approach to understanding the Rouquier dimension of $\mathcal{C}$. The first is producing a short resolution of the diagonal. For this, we take inspiration from symplectic geometry. In \cite{hanlon2023relating}, it was shown that when $\check X$ is a Weinstein domain, short resolutions of the diagonal of the wrapped Fukaya category $\mathcal{W}(\check X)$ can be constructed in the following fashion. Let $\core_X$ denote a skeleton of $\check X$ and let $\phi_H: \check X\to \check X$ be the time 1-Hamiltonian flow of $H: \check X\to \RR$. If the intersection $\phi_H(\core_X)\cap \core_X$ is \emph{transverse}, the diagonal bimodule for $\mathcal{W}(\check X)$ has a resolution by product objects of length $|H(\phi_H(\core_X)\cap \core_X)|$. Furthermore, the terms of the resolution can be read off from the intersection points $\phi_H(\core_X)\cap \core_X$. One excessively roundabout way to obtain \cref{eq:resDiagA1} is to use this strategy on the mirror space (see \cref{fig:skelA1a}). By translating back across the mirror, we see that \cref{eq:resDiagA1} has the structure of a cellular resolution (see \cref{fig:resA1a}). In \cite{hanlon2023resolutions}, this translation method was used to produce a short resolution of the diagonal of a smooth toric variety by line bundles.

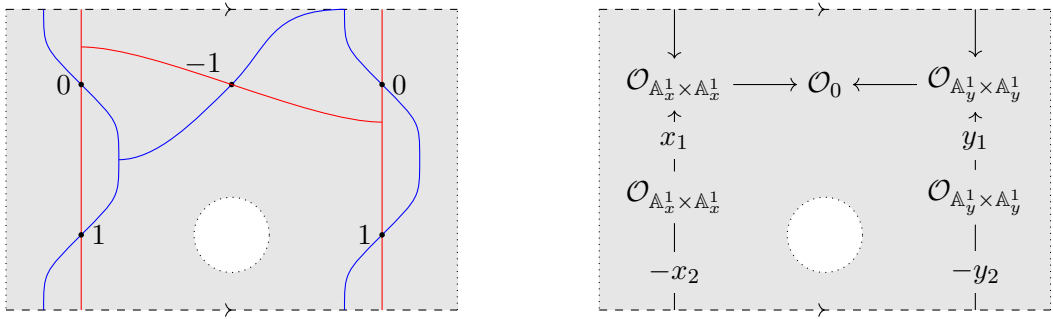
\begin{figure}[h]
	\centering
	\begin{subfigure}[t]{.47\linewidth}
		\centering
		\begin{tikzpicture}

\begin{scope}[decoration={    markings,
    mark=at position 0.5 with {\arrow{>}}}    ] 

\fill[gray!20]  (1.5,-3) rectangle (-4.5,1);

\draw[dashed, postaction={decorate}] (-4.5,1) -- (1.5,1);
\draw[dashed, postaction={decorate}](-4.5,-3) -- (1.5,-3);

\draw[dotted] (-4.5,1) -- (-4.5,-3);

\draw[dotted] (1.5,-3) -- (1.5,1) ;
\clip  (1.5,-3) rectangle (-4.5,1);

\draw[dotted, fill=white]  (-1.5,-2) ellipse (0.5 and 0.5);

\end{scope}

\begin{scope}[shift={(4,0)}]

\draw[blue] (-4,1) .. controls (-4,0.5) and (-4,0.5) .. (-3.5,0) .. controls (-3,-0.5) and (-3,-0.5) .. (-3,-1) .. controls (-3,-1.5) and (-3,-1.5) .. (-3.5,-2) .. controls (-4,-2.5) and (-4,-2.5) .. (-4,-3);

\end{scope}

\draw[blue] (-4,1) .. controls (-4,0.5) and (-4,0.5) .. (-3.5,0) .. controls (-3,-0.5) and (-3,-0.5) .. (-3,-1) .. controls (-3,-1.5) and (-3,-1.5) .. (-3.5,-2) .. controls (-4,-2.5) and (-4,-2.5) .. (-4,-3);

\draw[red] (-3.5,-3) -- (-3.5,1) (0.5,-3) -- (0.5,1) (-3.5,0.5);
\draw[blue] (-3,-1) .. controls (-2.5,-1) and (-2,-0.5) .. (-1.5,0) .. controls (-1,0.5) and (-1,1) .. (0,1);

\node[circle, fill, scale=.2] at (-3.5,0) {};
\node[circle, fill, scale=.2] at (-3.5,-2) {};

\node[circle, fill, scale=.2] at (0.5,0) {};

\node[circle, fill, scale=.2] at (-1.5,0) {};
\node[circle, fill, scale=.2] at (0.5,-2) {};
\node[right] at (-3.5,-2) {$1$};
\node[left] at (0.5,-2) {$1$};
\node[right] at (0.5,0) {$0$};
\node[left] at (-3.5,0) {$0$};
\node[above left] at (-1.5,0) {$-1$};
\draw[red] (-3.5,0.5) .. controls (-2.5,0.5) and (-0.5,-0.5) .. (0.5,-0.5);
\end{tikzpicture}
 		\caption{Perturbed skeleta of the pair of pants.}
		\label{fig:skelnodalConicSlow}
	\end{subfigure}
	\begin{subfigure}[t]{.47\linewidth}
		\centering
		\begin{tikzpicture}

\begin{scope}[decoration={    markings,
    mark=at position 0.5 with {\arrow{>}}}    ] 

\fill[gray!20]  (1.5,-3) rectangle (-4.5,1);

\draw[dashed, postaction={decorate}] (-4.5,1) -- (1.5,1);
\draw[dashed, postaction={decorate}](-4.5,-3) -- (1.5,-3);

\draw[dotted] (-4.5,1) -- (-4.5,-3);

\draw[dotted] (1.5,-3) -- (1.5,1) ;
\clip  (1.5,-3) rectangle (-4.5,1);

\node (v5) at (-4.5,-1.5) {};
\begin{scope}[]
\node (v2) at (-3.5,0) {$\mathcal O_{\mathbb A_x^1\times \mathbb A_x^1}$};
\node (v3) at (-3.5,-1.5) {$\mathcal O_{\mathbb A^1_x\times \mathbb A^1_x}$};
\draw   (-3.5,-3) edge (v3);
\draw  (v2) edge[<-] (-3.5,1);
\node[fill=gray!20] at (-3.5,-2.5) {$-x_2$};

\end{scope}

\draw (v3) edge[->] node[fill=gray!20] {$x_1$}(v2);
\node (v6) at (-1.5,0) {$\mathcal O_{0}$};

\draw[<-]  (v6) edge (v2);
\begin{scope}[shift={(4,0)}]
\node (v2) at (-3.5,0) {$\mathcal O_{\mathbb A^1_y \times \mathbb A^1_y}$};
\node (v3) at (-3.5,-1.5) {$\mathcal O_{\mathbb A^1_y\times \mathbb A^1_y}$};
\draw   (-3.5,-3) edge (v3);
\draw  (v2) edge[<-] (-3.5,1);
\node[fill=gray!20] at (-3.5,-2.5) {$-y_2$};

\end{scope}

\draw (v3) edge[->] node[fill=gray!20] {$y_1$}(v2);
\draw[<-]  (v6) edge (v2);

\draw[dotted, fill=white]  (-1.5,-2) ellipse (0.5 and 0.5);

\end{scope}
\end{tikzpicture}
 		\caption{Diagonal resolution of $V(xy)$ by prod. objects.}
		\label{fig:resNodalConicSlow}
	\end{subfigure}
	\caption{The pair of pants is homologically mirror to the singular conic.}
 \label{fig:nodalConicSlow}
\end{figure}
The second difficulty is that product objects are slightly too rigid to quickly generate the diagonal in some cases. For example, consider the symplectic pair of pants, which is homologically mirror to $V(xy)$, the nodal conic. The shortest resolution of the diagonal by product bimodules we could find by this method is exhibited in \cref{fig:nodalConicSlow}, which does no better than the trivial bound one obtains from pulling back the resolution of the diagonal for the space $\AA^1_x\times \AA^1_y$!
\begin{figure}[h]
	\centering
	\begin{subfigure}[t]{.4\linewidth}
		\centering
		\begin{tikzpicture}

    \begin{scope}[decoration={
        markings,
        mark=at position 0.5 with {\arrow{>}}}
        ] 
    
    \fill[gray!20]  (-1.5,-3) rectangle (-4.5,1);
    
    \draw[dashed, postaction={decorate}] (-4.5,1) -- (-1.5,1);
    \draw[dashed, postaction={decorate}](-4.5,-3) -- (-1.5,-3);
    
    \draw[dotted] (-4.5,1) -- (-4.5,-3);
    
    \draw (-1.5,-1.5) -- (-1.5,1) (-1.5,-3) -- (-1.5,-2.5);
    \clip  (-1.5,-3) rectangle (-4.5,1);

    \draw[dotted, fill=white]  (-1.5,-2) ellipse (0.5 and 0.5);

    \draw[blue] (-3.75,1) .. controls (-3.75,0.5) and (-3.75,0.5) .. (-3.5,0) .. controls (-3.25,-0.5) and (-3.25,-0.5) .. (-3.25,-1) .. controls (-3.25,-1.5) and (-3.25,-1.5) .. (-3.5,-2) .. controls (-3.75,-2.5) and (-3.75,-2.5) .. (-3.75,-3);

    \end{scope}
    \draw[red] (-3.5,1) -- (-3.5,-3);
    \draw[red] (-3.5,0) .. controls (-3,0) and (-2.5,0) .. (-1.5,0);
    \draw[blue] (-3.5,0) .. controls (-3,0.5) and (-2.5,0.5) .. (-2,0) .. controls (-1.5,-0.5) and (-2.5,-2) .. (-2,-2);
    \node[fill=black, scale=.2, circle] at (-3.5,-2) {};
\node[fill=black, scale=.2, circle] at (-3.5,0) {};
\node[fill=black, scale=.2, circle] at (-2,0) {};
\node[left] at (-3.5,-2) {$1$};
\node[left] at (-3.5,0) {$0$};
\node[above right] at (-2,0) {$1$};
\end{tikzpicture} 		\caption{A perturbation of the skeleton whose intersections are not transverse.}
		\label{fig:skelA1b}
	\end{subfigure}
	\hspace{.1\linewidth}
	\begin{subfigure}[t]{.4\linewidth}
		\centering
		\begin{tikzpicture}

\begin{scope}[decoration={    markings,
    mark=at position 0.5 with {\arrow{>}}}    ] 

\fill[gray!20]  (-1.5,-3) rectangle (-4.5,1);

\draw[dashed, postaction={decorate}] (-4.5,1) -- (-1.5,1);
\draw[dashed, postaction={decorate}](-4.5,-3) -- (-1.5,-3);

\draw[dotted] (-4.5,1) -- (-4.5,-3);

\draw (-1.5,-1.5) -- (-1.5,1) (-1.5,-3) -- (-1.5,-2.5);
\clip  (-1.5,-3) rectangle (-4.5,1);
\def\int{1}

\node (v5) at (-4.5,-1.5) {};
\node (v2) at (-3.5,0) {$\mathcal I_0$};
\node (v3) at (-3.5,-1.75) {$\mathcal O_{\mathbb A^1\times \mathbb A^1}$};

\node (v6) at (-2,0) {$\mathcal O_{0}[1]$};
\draw (v3) edge[->] node[fill=gray!20] {$\phi^{x_1}$}(v2);
\draw   (-3.5,-3) edge (v3);
\draw[->]  (v6) edge  node[fill=gray!20]{$\delta$} (v2);

\draw  (v2) edge[<-] (-3.5,1);

\draw[dotted, fill=white]  (-1.5,-2) ellipse (0.5 and 0.5);

\node[fill=gray!20] at (-3.5,-2.5) {$-\phi^{x_2}$};
\end{scope}
\end{tikzpicture} 		\caption{Generation of the diagonal in $\AA^1\times \AA^1$ by weakly product bimodules}
		\label{fig:resA1b}
	\end{subfigure}
	\caption{When $\phi_H(\core_X)$ and $\core_X$ do not intersect transversely, we still explicitly generate the diagonal (but not by product objects). The generation is good enough to bound $\Rdim(\Fuk(\check X))$.}
	\label{fig:A1b}
\end{figure}
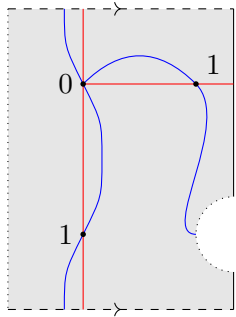
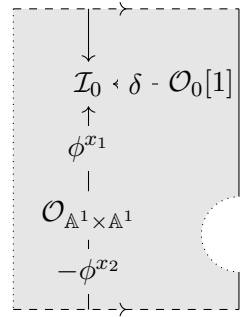
To get around this, we again take inspiration from symplectic geometry and attempt to apply  \cite{hanlon2023relating} to perturbations that result in non-transverse intersections such as the perturbation drawn in \cref{fig:skelA1b}. In good cases, the results of \cite{hanlon2023relating} imply that these perturbations correspond to resolutions of the diagonal, but not necessarily into product objects. There is not an obvious way to characterize these objects symplectically.

Here, we take this inspiration and use it to perform computations in algebraic geometry in a basic singular example. While we do use intuition from symplectic geometry to exhibit a complex of objects in $D^b\Coh(I_n)$, the actual check that this complex (as in \cref{fig:resA1b}) is a short derived resolution of the diagonal is a purely algebraic computation. In our examples, there is a natural class of objects that correspond to these non-transverse intersections, which behave like product objects for the aspect of bounding Rouquier dimension.

\section{Weakly product bimodules and diagonal generation}
We now introduce weakly product bimodules and show that the generation time of the diagonal by weakly product bimodules bounds the Rouquier dimension.
\label{sec:weakDiagonalDimension}
\subsection{Background: Rouquier and diagonal dimension}
We use the notation in \cite{hanlon2023relating}: let $\mathcal C$ be a dg (or $A_\infty$) category over a field $k$, and let $\Perf(\mathcal C)$ be the idempotent-completed pre-triangulated closure of $\mathcal C$. All functors in this paper are implicitly derived. We will abuse notation and use an object $G$ for the full subcategory on that object.

Following \cite{rouquier2008dimensions}, given subcategories $\mathcal G_1,\mathcal  G_2$ of $\Perf \mathcal C$, let
$\mathcal{G}_1 * \mathcal{G}_2$ be the full subcategory of objects $G_3=\cone(G_1\to G_2)$ with $G_1\in \mathcal G_1, G_2\in \mathcal  G_2$.
The idempotent closure of $\mathcal{G}\subset \Perf \mathcal{C}$, denoted by $\langle \mathcal{G} \rangle$, is the smallest full subcategory containing $\mathcal G$ and closed under quasi-isomorphisms, direct summands, shifts, and finite direct sums. Let $\ads(\mathcal G)$ be the smallest full subcategory containing $\mathcal G$ and closed under (possibly infinite) direct sums, summands, and shifts.

Given a full subcategory $\mathcal G$ of $\Perf \mathcal C$, we inductively build full subcategories
\[\langle \mathcal G \rangle_0, \langle \mathcal G \rangle_1, \langle G \rangle_2, \ldots \]
by the following process.
For the base case, set:
$\langle \mathcal G \rangle_{0}=0$
and subsequently define
\[\langle \mathcal{G} \rangle_k = \langle \langle \mathcal{G} \rangle_{k-1} * \langle \mathcal{G} \rangle \rangle.\]
We then define the generation time of $\mathcal E\subset\Perf(\mathcal C)$ by $\mathcal G$  to be
\[\gentime_{\mathcal{G}} (\mathcal{E}) = \min \left(\left\{ k\in \NN \st \mathcal{E} \subset \langle \mathcal G \rangle_{k+1} \right\}\cup\{+\infty\} \right).\]
\begin{definition}[Rouquier dimension] \label{def:rdim}
	Let $\mathcal C$ be a dg (or $A_\infty$) category over a field $k$. The \emph{Rouquier dimension} of $\mathcal C$ is
	\[\Rdim(\mathcal C):=\min\left( \{k \st \exists G\in \Perf(\mathcal C) , \gentime_G(\mathcal C)=k\}\cup \{+\infty\}\right).\]
\end{definition}

Let $\mathcal B$ and $\mathcal C$ be dg categories. Given a bimodule $K\in \mathcal B^{op}\times \mathcal C$ and $F\in \mathcal B$ we define the convolution
\begin{align*}
	\Phi_K: \mathcal B\to \mathcal C &  & F\mapsto K\tensor_{\mathcal B} F.
\end{align*}
Denote by $\Delta\in \mathcal C^{op}\otimes \mathcal C$ the diagonal bimodule $\Delta(F_1, F_2) = \Hom^\bullet(F_1, F_2)$; this satisfies the property that $\Phi_\Delta F\cong F$. For objects $G_1, G_2\in \mathcal C$, denote by $G_1\boxtimes G_2$ the product bimodule. The associated convolution is
\begin{equation}
	\Phi_{G_1\boxtimes G_2}(F)\cong H^\bullet(G_1\tensor F)\tensor G_2.
	\label{eq:prodIsWeakProd}
\end{equation}
Let $\mathcal P\subset \mathcal C^{op}\otimes \mathcal C$ denote the idempotent closure of the full subcategory of product bimodules.

\begin{definition} \label{def:ddim}
	The diagonal dimension of a dg category $\mathcal C$ is
	\[\Ddim(\mathcal C):=\min \left( \{ k \st \exists G \in \mathcal P, \gentime_G(\Delta)=k\} \cup \{+\infty\} \right).\]
\end{definition}
It is well-established that  $\Rdim(\mathcal C) \leq \Ddim(\mathcal C)$ (see, for instance, \cite{elagin2021three}).
\subsection{Weakly product bimodules}

Requiring the objects in $\mathcal{C}^{op} \times \mathcal{C}$ that generate the diagonal to be product bimodules is stronger than necessary for applying the usual argument to bound the Rouquier dimension.
\begin{definition}
	\label{def:weaklyProduct}
	We say that $ G_{12}\in \mathcal C^{op}\times \mathcal C$ is a \emph{weakly product bimodule} if there exists an object $ G_2\in \mathcal C$ so that for all $ F \in \mathcal C$ we have
	\[   \Phi_{G_{12}}( F) \in  \ads(  G_2).\]
\end{definition}
Observe that every product bimodule $ G_{12}= G_1\boxtimes  G_2$ is a weakly product bimodule by \cref{eq:prodIsWeakProd} and that summands and shifts of weakly product bimodules are weakly product bimodules.
For this reason, it makes sense to denote by $\mathcal {WP}\subset \mathcal C^{op}\otimes \mathcal C$ the idempotent closed subcategory of weakly product bimodules, which contains $\mathcal P$. We will now construct some objects of $\mathcal {WP}$ from certain mapping cones in $\mathcal P$.
\begin{lemma}
	Let $G_1\boxtimes G_2$ and $G'_1\boxtimes G'_2$ be product bimodules.
	Let $\phi_i: G_i\to G'_i$ be morphisms. Then $G''_{12}:= \cone(\phi_1\boxtimes \phi_2)$ is weakly product.
	\label{lem:makingWeaklyProduct}
\end{lemma}
\begin{proof}
	Using that all of our functors are triangulated, for any object $F\in \mathcal C$
	\begin{align*}
		\Phi_{G''_{12}}(F)\cong & \cone\left(H^\bullet(G_1\tensor F)\tensor G_2\xrightarrow{H^\bullet(\phi_1\tensor \id )\tensor \phi_2}H^\bullet(G'_1\tensor F)\tensor G'_2\right) \\
		\intertext{Since $H^\bullet(\phi_1\tensor \id )$ is a graded map of graded vector spaces, we can row reduce to obtain decompositions $H^\bullet(G_1\tensor F)=V\oplus U$ and $H^\bullet(G'_1\tensor F)=V\oplus W$  so that
		}
		\Phi_{G''_{12}}(F)\cong                   & \cone\left((V\oplus U)\tensor G_2 \xrightarrow{\begin{pmatrix}\id_V\tensor \phi_2 & 0\\0& 0 \end{pmatrix}} (V\oplus W)\tensor G'_2\right)         \\
		\cong                   & (U\tensor G_2)\oplus (V\tensor \cone(\phi_2)) \oplus (W\tensor G'_2).
	\end{align*}
	Therefore the image of $\Phi_{G''_{12}}$ is contained within $\ads(G''_2)$ where $G''_2=G_2\oplus \cone(\phi_2)\oplus G'_2$.
\end{proof}
\begin{example}
	There are more ways to build weakly product bimodules than \cref{lem:makingWeaklyProduct}.
	For example, suppose that we have the following commutative diagram:
	\[
		\begin{tikzcd}
			G_1\otimes G_2 \arrow{r}{g_1\otimes g_2} \arrow{d}{\phi_1\otimes \phi_2} & G_1'\otimes G_2' \arrow{d}{\phi_1'\otimes \phi_2'}  \arrow{r} & \cone(g_1\otimes g_2) \arrow{d}{\phi}\\
			H_1\otimes H_2 \arrow{r}{h_1\otimes h_2} &  H_1'\otimes H_2' \arrow{r} & \cone(h_1\otimes h_2)\arrow{d}\\ 
			 & & \cone(\phi)
		\end{tikzcd}
	\]
Then, the object $\cone(\phi)$ is weakly product by a similar argument to \cref{lem:makingWeaklyProduct}.
\end{example}
\subsection{Bounding the Rouquier dimension}
Weakly product bimodules remain useful for bounding the generation time.
\begin{thm}
	\label{thm:weakProductToRdim}
	Let $G_{12}$ be a weakly product bimodule and suppose $\gentime_{G_{12}}(\Delta)=k$. Then $\gentime_{G_2}(\mathcal C)\leq k$. In particular, $\Rdim(\mathcal C)\leq k$. 
\end{thm}
\begin{proof}
	The proof is a minor generalization of the statement that $\Rdim(\mathcal C)\leq \Ddim(\mathcal C)$. We prove that any $F$ is generated in time $k$ by $G_2$. Because $\gentime_{G_{12}}(\Delta)=k$, there exists a sequence
	\[0\to R_0\to R_1\to \cdots \to R_k\]
	with $\Delta$ a summand of $R_k$ and $G_{12}^i=\cone(R_{i-1}\to R_i)$ is an element of $\langle G_{12}\rangle$.
	Applying $\Phi_-(F)$ to this sequence (and using that the convolution functor is triangulated in both factors) yields:
	\[0 \to \Phi_{R_0}(F)\to \Phi_{R_1}(F)\to \cdots \to \Phi_{R_k}(F),\]
	where $F \cong \Phi_{\Delta}(F)$ is a direct summand of $\Phi_{R_k}(F)$. Each $\cone(\Phi_{R_{i-1}}(F)\to \Phi_{R_i}(F))$ is isomorphic to $\Phi_{G_{12}^i}(F)$. Since each $G_{12}^i$ is a direct summand of the weakly product bimodule $G_{12}$, $\Phi_{G_{12}^i}(F)\in \ads( G_2)$. This shows that $F\in \langle \ads(G_2)\rangle_k$. From \cite[Corollary 3.14]{rouquier2008dimensions}, it follows that $F\in \langle G_2\rangle_k$.
\end{proof}

\section{Examples}
In this section, we generate the diagonals of $\AA, V(xy),$ and $I_n$ with weakly product bimodules.
\label{sec:examples}
\subsection{Toy case: the affine line}
Let $\mathcal O_{\AA\times \AA}=R=k[x_1, x_2]$. Let $\mathcal O_0 = R^1/(x_1, x_2)$ be the skyscraper sheaf of the origin, and let $\mathcal I_0 = \cone(\mathcal O_{\AA\times \AA}\to \mathcal O_0)[-1]$ be the ideal sheaf of the origin. By \cref{lem:makingWeaklyProduct}, $\mathcal{I}_0$ is weakly product. We have morphisms (in the derived category) $\phi^{x_1}, \phi^{x_2}:\mathcal O_{\AA\times \AA}\to \mathcal I_0$  and $\delta: \mathcal O_0[-1]\to \mathcal I_0$ described in the complex below:
\[
	\begin{tikzpicture}
		\node (v2) at (4.5,2) {$R^1$};
		\node (v5) at (1,-1) {$R^3$};
		\node (v6) at (-3,-1) {$R^1$};
		\node at (1,3) {0};
		\node (v4) at (4.5,-1) {$R^1$};
		\node at (4.5,3) {1};
		\node at (8,3) {2};
		\node at (-3,3) {-1};
		\node (v12) at (-5.5,2) {$\mathcal O_{\mathbb A^1\times \mathbb A^1}$};
		\node (v11) at (-7.5,-1) {$\mathcal I_0$};
		\draw[->]  (v4) edge node[fill=white, scale=.75]{$\begin{pmatrix}0\\x_1\\-x_2\end{pmatrix}$} (v5);
		\draw[->]  (v5) edge node[fill=white, scale=.75]{$\begin{pmatrix} 1&x_1&x_2\end{pmatrix}$}(v6);
		\draw[->]  (v2) edge node[fill=white, scale=.75]{$\begin{pmatrix}x_1\\-1\\0\end{pmatrix}-\begin{pmatrix}x_2\\0\\-1\end{pmatrix}$}(v5);
		\node (v8) at (1,-3.5) {$R^1$};
		\node (v7) at (4.5,-3.5) {$R^2$};
		\node (v9) at (8,-3.5) {$R^1$};
		\draw  (v7) edge[->] node[fill=white, scale=.75]{$\begin{pmatrix} x_1& x_2\end{pmatrix}$} (v8);
		\draw  (v9) edge[->]  node[fill=white, scale=.75]{$\begin{pmatrix} x_2\\ -x_1\end{pmatrix}$}(v7);
		\draw  (v8) edge[->]  node[fill=white, scale=.75]{$\begin{pmatrix} 1\end{pmatrix}$}(v6);
		\draw  (v7) edge[->] node[fill=white, scale=.75]{$\begin{pmatrix} 0&0\\1&0\\0&1\end{pmatrix}$} (v5);
		\draw  (v9) edge[->] node[fill=white, scale=.75]{$\begin{pmatrix} 1\end{pmatrix}$} (v4);
		\node (v10) at (-5.5,-3.5) {$\mathcal O_0[-1]$};
		\draw  (v10) edge[->] node[fill=white, scale=.75]{$\delta$} (v11);
		\draw  (v12) edge[->] node[fill=white, scale=.75]{$\phi^{x_1}-\phi^{x_2}$} (v11);
		\draw[dashed] (-4,3) -- (-4,-4);
	\end{tikzpicture}
\]
This is the complex exhibited in \cref{fig:resA1b} and agrees with the perturbation of the skeleton from \cref{fig:skelA1b}. By presenting our complex as a free resolution shown on the right, we can use Macaulay2 to verify that this is quasi-isomorphic to the diagonal.
Therefore, we've shown that the diagonal of $\AA^1$ is generated in time 1 by weakly product bimodules, with generation time minimized by the structure sheaf of $\AA^1$, the skyscraper sheaf at the origin, and the ideal sheaf of the origin (in this case isomorphic to the structure sheaf).

\subsection{The singular conic}

\label{exam:conic}
Let $X=\Spec(k[x_1,y_1]/(x_1y_1))$ be the nodal conic, and let $R=k[x_1, y_1, x_2, y_2]/(x_1y_1, x_2y_2)$ be the product.  We resolve the diagonal in $X\times X = \Spec(R)$. To slightly simplify notation, we write $\mathcal{O}_{\AA_x \times \AA_x}$ for the structure sheaf of $\AA_{x_1} \times \AA_{x_2}$ on $X \times X$. Let:
\begin{align*}
	\mathcal O_{\AA_x\times \AA_x}= R^1/(y_1, y_2), \mathcal O_{\AA_y\times \AA_y}=R^1/(x_1, x_2), \mathcal O_0=R^1/(x_1, y_1, x_2, y_2).
\end{align*}
The following complex whose leftmost term is in homological degree $-1$
\[\mathcal O_0\xleftarrow{\begin{pmatrix} 1& 1\end{pmatrix}}\mathcal O_{\AA_x\times \AA_x}\oplus \mathcal O_{\AA_y\times \AA_y} \xleftarrow{\begin{pmatrix}x_1-x_2 & 0\\0 & y_1-y_2\end{pmatrix}} \mathcal O_{\AA_x\times \AA_x}\oplus \mathcal O_{\AA_y\times \AA_y}
\]
is quasi-isomorphic the diagonal. It is the complex exhibited in \cref{fig:nodalConicSlow} from the introduction.
This complex is homotopic to the complex depicted in \cref{fig:resolutionNodalConicMedium}. We now replace the subcomplexes spanned by the target and domain of the red arrows in \cref{fig:resolutionNodalConicMedium} by the objects they represent in the derived category.  Let $\mathcal I^x_0= \cone(\mathcal O_{\AA_x\times \AA_x}\to \mathcal O_0)[-1]$, and similarly let $\mathcal I^y_0= \cone(\mathcal O_{\AA_y\times \AA_y}\to \mathcal O_0)[-1]$; these are weakly product bimodules by \cref{lem:makingWeaklyProduct}. As the map $(x_i): \mathcal O_{\AA_x\times \AA_x}\to \mathcal O_{\AA_x\times \AA_x}$ composes to zero with the red arrow, we obtain a morphism in the derived category $\phi^{x_i}: \mathcal O_{\AA_x\times \AA_x} \to \mathcal I^x_0$; similarly, we have morphisms
\begin{align*}
	\phi^{x_1}, \phi^{x_2}:\mathcal O_{\AA_x\times \AA_x}\to \mathcal I^x_0 &  & \phi^{y_1}, \phi^{y_2}:\mathcal O_{\AA_y\times \AA_y}\to \mathcal I^y_0 \\
	\delta^x: \mathcal O_0[-1]\to \mathcal I^x_0                                &  & \delta^y: \mathcal O_0[-1]\to \mathcal I^y_0.
\end{align*}
\begin{lemma}
	\label{lem:resolutionNodalConic}
The following complex in the derived category
\[\mathcal I^x_0[-1]\oplus \mathcal I^y_0[-1] \xleftarrow{\begin{pmatrix}\phi^{x_1}-\phi^{x_2}&\delta^x& 0 \\ 0 & \delta^y &  \phi^{y_1}-\phi^{y_2}\end{pmatrix}} \mathcal O_{\AA_x\times \AA_x}\oplus \mathcal O_0 \oplus   \mathcal O_{\AA_y\times \AA_y}\]
is a length 1 resolution of $\mathcal O_{\Delta}$ by weakly product bimodules. 
\end{lemma}
To make connection with our geometric intuition so far, we have drawn the resolution superimposed on the pair of pants in \cref{fig:resolutionNodalConic}. As suggested by the figure, the pullback of this resolution  to $\AA_x\times \AA_x$ agrees with \cref{fig:resA1b}.
\begin{figure}[h]
	\centering
	\begin{subfigure}{.45\linewidth}
		\centering
		\begin{tikzpicture}

\begin{scope}[decoration={    markings,
    mark=at position 0.5 with {\arrow{>}}}    ] 

\fill[gray!20]  (1.5,-3) rectangle (-4.5,1);

\draw[dashed, postaction={decorate}] (-4.5,1) -- (1.5,1);
\draw[dashed, postaction={decorate}](-4.5,-3) -- (1.5,-3);

\draw[dotted] (-4.5,1) -- (-4.5,-3);

\draw[dotted] (1.5,-3) -- (1.5,1) ;
\clip  (1.5,-3) rectangle (-4.5,1);

\node (v5) at (-4.5,-1.5) {};
\begin{scope}[]
\node (v2) at (-3.5,0.5) {$\mathcal I^x_0$};
\node (v3) at (-3.5,-1.5) {$\mathcal O_{\mathbb A^1_x\times \mathbb A^1_x}$};
\draw   (-3.5,-3) edge (v3);
\draw  (v2) edge[<-] (-3.5,1);
\node[fill=gray!20] at (-3.5,-2.5) {$-\phi^{x_2}$};

\end{scope}

\draw (v3) edge[->] node[fill=gray!20] {$\phi^{x_1}$}(v2);
\node (v6) at (-1.5,0.5) {$\mathcal O_{0}[1]$};

\draw[->]  (v6) edge (v2);
\begin{scope}[shift={(4,0)}]
\node (v2) at (-3.5,0.5) {$\mathcal I^y_0$};
\node (v3) at (-3.5,-1.5) {$\mathcal O_{\mathbb A^1_y\times \mathbb A^1_y}$};
\draw   (-3.5,-3) edge (v3);
\draw  (v2) edge[<-] (-3.5,1);
\node[fill=gray!20] at (-3.5,-2.5) {$-\phi^{y_2}$};

\end{scope}

\draw (v3) edge[->] node[fill=gray!20] {$\phi^{y_1}$}(v2);
\draw[->]  (v6) edge (v2);

\draw[dotted, fill=white]  (-1.5,-2) ellipse (0.5 and 0.5);

\end{scope}
\end{tikzpicture}
 		\caption{A weakly product resolution of the diagonal in the nodal conic}
		\label{fig:resolutionNodalConic}
	\end{subfigure}
	\begin{subfigure}{.45\linewidth}
		\centering
		\begin{tikzpicture}

\begin{scope}[decoration={    markings,
    mark=at position 0.5 with {\arrow{>}}}    ] 

\fill[gray!20]  (1.5,-3) rectangle (-4.5,1);

\draw[dashed, postaction={decorate}] (-4.5,1) -- (1.5,1);
\draw[dashed, postaction={decorate}](-4.5,-3) -- (1.5,-3);

\draw[dotted] (-4.5,1) -- (-4.5,-3);

\draw[dotted] (1.5,-3) -- (1.5,1) ;
\clip  (1.5,-3) rectangle (-4.5,1);

\node (v5) at (-4.5,-1.5) {};
\begin{scope}[]
\node (v2) at (-3.5,0.5) {$\mathcal O_{\mathbb A^1_x\times \mathbb A^1_x}$};
\node (v3) at (-3.5,-1.5) {$\mathcal O_{\mathbb A^1_x\times \mathbb A^1_x}$};
\draw   (-3.5,-3) edge (v3);
\draw  (v2) edge[<-] (-3.5,1);
\node[fill=gray!20] at (-3.5,-2.5) {$-x_2$};

\end{scope}

\draw (v3) edge[->] node[fill=gray!20] {$x_1$}(v2);
\node (v6) at (-2.5,-0.5) {$\mathcal O_{0}$};
\node (v7) at (-1.5,0.5) {$\mathcal O_{0}$};
\node (v8) at (-0.5,-0.5) {$\mathcal O_{0}$};

\draw[<-,red]  (v6) edge (v2);
\begin{scope}[shift={(4,0)}]
\node (v2) at (-3.5,0.5) {$\mathcal O_{\mathbb A^1_y\times \mathbb A^1_y}$};
\node (v3) at (-3.5,-1.5) {$\mathcal O_{\mathbb A^1_y\times \mathbb A^1_y}$};
\draw   (-3.5,-3) edge (v3);
\draw  (v2) edge[<-] (-3.5,1);
\node[fill=gray!20] at (-3.5,-2.5) {$-y_2$};

\end{scope}

\draw (v3) edge[->] node[fill=gray!20] {$y_1$}(v2);
\draw[<-,red]  (v8) edge (v2);

\draw[dotted, fill=white]  (-1.5,-2) ellipse (0.5 and 0.5);

\end{scope}
\draw[->]  (v7) edge (v6);
\draw[->]  (v7) edge (v8);
\end{tikzpicture}
 		\caption{A less efficient resolution of the diagonal in the nodal conic.}
		\label{fig:resolutionNodalConicMedium}
	\end{subfigure}
	\caption{Resolutions of the diagonal on $X\times X$ drawn on a pair of pants homologically mirror to $X = V(xy)$.}
\end{figure}
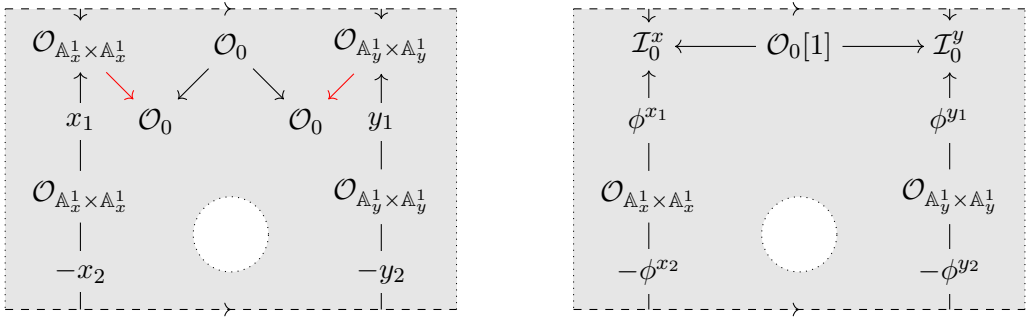
\begin{proof}[Proof of \cref{lem:resolutionNodalConic}]
The complex is homotopic to \cref{fig:resolutionNodalConicMedium}, whose pullback $\AA_x\times \AA_x, \AA_x\times \AA_y, \AA_y\times \AA_x$, and $\AA_y\times \AA_y$ resolves the restriction of $\mathcal O_{\Delta}$ to each of those four charts.
\end{proof}
\subsection{The cycle of projective lines}
Consider projective lines $ \PP^1_j$ for $j=1, \ldots, n$ with homogeneous coordinates $[x_{j,-}: x_{j,+}]$ and let $p_{j,-}, p_{j,+}\in \PP^1_j$ correspond to the points $0, \infty$. Consider the curve \[I_n=\left(\bigcup_{i=1}^n \PP^1_j \right)/(p_{i,-}\sim p_{i,+}),\] where all indices are taken cyclicly modulo $n$. Let $\mathcal O_{p_j}$ be the skyscraper sheaf at the point $p_j=((p_{j,+})_1,(p_{j,+})_2)\in I_n\times I_n$.
The space  $\hom_{\PP_j\times \PP_j}(\mathcal O_{\PP_j\times \PP_j}(-1,-1), \mathcal O_{\PP_j\times \PP_j})$ is generated by monomial sections $\{(x_{j,+})_1,(x_{j,+})_2,(x_{j,-})_1,(x_{j,-})_1\}$.
 By abuse of notation, denote by  $\mathcal O_{\PP_j^1\times \PP_j^1}(d_1, d_2)$ the sheaf $(i_j)_*\mathcal O_{\PP_j^1}(d_1)\boxtimes  (i_j)_*\mathcal O_{\PP_j^1}(d_2)$, where $i_j: \PP^1_j\to I_n$. From this, define the sheaves
\[\mathcal F_j = (i_j)_*\cone(\mathcal O_{\PP_j^1\times \PP_j^1}\to  \mathcal O_{p_j})[-1]\]
\[\mathcal G_j = (i_j)_*\cone(\mathcal O_{\PP_j^1\times \PP_j^1}\to \mathcal O_{p_{j-1}})[-1]\]
Define the maps 
\begin{align*}
	\phi^{(x_{j,+})_1},\phi^{(x_{j,+})_2}: \mathcal O_{\PP_j^1\times \PP_j^1}(-1, -1)\to \mathcal F_j && \phi^{(x_{j,+})_2},\phi^{(x_{j,-})_1}: \mathcal O_{\PP_j^1\times \PP_j^1}(-1, -1)\to \mathcal G_j \\
	\delta_j^+:\mathcal O_{p_j}[1]\to \mathcal G_{j+1} && \delta_j^-:\mathcal O_{p_j}[1]\to \mathcal G_{j}
\end{align*}
as before. 
\begin{thm}
\label{thm:resIn}
The following complex in the derived category
\[
\bigoplus_{j=1}^n \mathcal F_k\oplus \mathcal G_k \xleftarrow{(A_{jk})}
\bigoplus_{k=1}^n\mathcal  \mathcal O_{\PP^1_j\times \mathbb P^1_j}(-1, -1)^{\oplus 2}\oplus \mathcal O_{p_j}[1]
\]
where
\[
A_{jk}=\left\{\begin{array}{cc}
	\begin{pmatrix}
	\phi^{(x_{j,+})_1} & - \phi^{(x_{j,+})_2} & \delta_j^+\\
	-\phi^{(x_{j,-})_1} &  \phi^{(x_{j,-})_2}& 0 
	\end{pmatrix} & \text{if $j=k$}\\
	\begin{pmatrix}
	0 & 0 & 0\\
	0 & 0 & \delta_j^- 
	\end{pmatrix} & \text{if $j=k+1$}\\
	0 & \text{otherwise}
\end{array}
\right.
\]
is a length 1 resolution of $\mathcal O_{\Delta}$ by weakly product bimodules. 
\end{thm}
To connect with our intuition from symplectic geometry, we've drawn the complex superimposed on the $n$-punctured torus in \cref{fig:resIn}
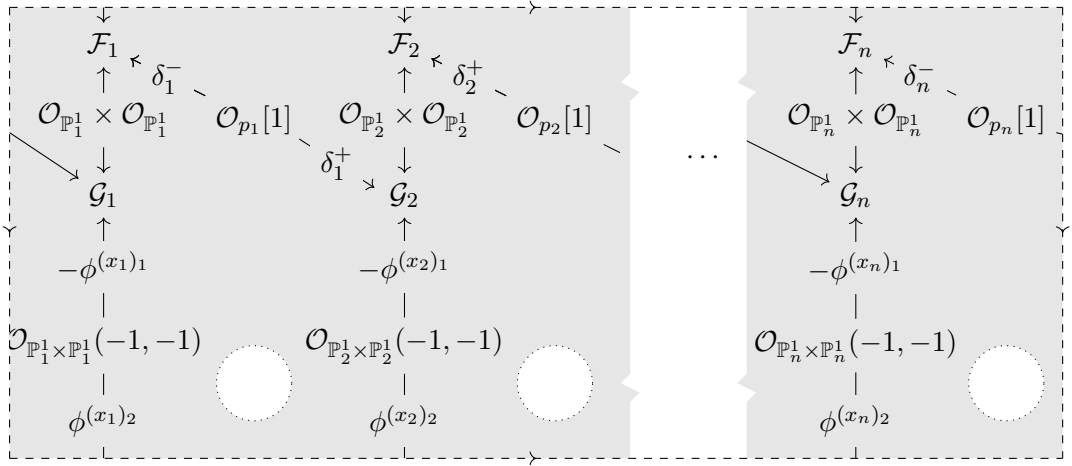
\begin{figure}[h]
\centering
\begin{tikzpicture}

\begin{scope}[shift={(1,-6.5)}]

\fill[gray!20] (3.5,2) -- (3.55,-0.375) -- (3.425,-0.5) -- (3.675,-0.625) -- (3.55,-0.75) -- (3.55,-4.875) -- (3.425,-5) -- (3.675,-5.125) -- (3.55,-5.25) -- (3.55,-6)-- (7.75,-6) -- (7.75,2) -- (3.5,2);

\clip (3.5,2) -- (3.55,-0.375) -- (3.425,-0.5) -- (3.675,-0.625) -- (3.55,-0.75) -- (3.55,-4.875) -- (3.425,-5) -- (3.675,-5.125) -- (3.55,-5.25) -- (3.55,-6)-- (7.75,-6) -- (7.75,2) -- (3.55,2);

\begin{scope}[shift={(8.5,0)}]

\def\int{n}

\node (v2) at (-3.5,1.5) {$\mathcal F_\int$};
\node (v3) at (-3.5,-4.5) {$\mathcal O_{\mathbb P^1_\int\times \mathbb P^1_\int}(-1,-1)$};
\draw   (-3.5,-6) edge (v3);
\draw  (v2) edge[<-] (-3.5,2);

\node (v6) at (-1.35,-0.5) {$\mathcal O_{p_\int}[1]$};

\draw[->]  (v6) edge node [fill=gray!20] {$\delta_\int^-$} (v2);

\node (v8) at (0.5,-2) {};
\draw[->]  (v6) edge node [fill=gray!20] {$\delta_\int^+$} (v8);

\draw[dotted, fill=white]  (-1.5,-5) ellipse (0.5 and 0.5);

\node (v4) at (-5.5,-1) {};

\node (v7) at (-3.5,-0.5) {$\mathcal O_{\mathbb P^1_\int\times \mathbb P^1_\int}(-1, -1)$};
\node (v5) at (-3.5,-2.5) {$\mathcal G_\int$};
\draw[->]  (v3) edge (v5);

\draw[->]  (v7) edge (v5);

\draw[->]  (v7) edge (v2);
\draw  (v4) edge[->] (v5);

\end{scope}

\end{scope}

\node at (4,-9) {$\cdots$};

\begin{scope}[decoration={    markings,
    mark=at position 0.5 with {\arrow{>}}}    , shift={(0,-6.5)}]

\fill[gray!20] (3,2) -- (3,0.125) -- (2.875,0) -- (3.125,-0.125) -- (3,-0.25) -- (3,-4.875) -- (2.875,-5) -- (3.125,-5.125) -- (3,-5.25) -- (3,-6)-- (-5.25,-6) -- (-5.25,2) -- (3,2);

\draw[dashed, postaction={decorate}] (-5.25,2) -- (8.75,2);
\draw[dashed, postaction={decorate}](-5.25,-6) -- (8.75,-6);

\draw[dashed, postaction={decorate}] (-5.25,2) -- (-5.25,-6);
\draw[dashed, postaction={decorate}] (8.75,2) -- (8.75,-6);

\clip (3,2) -- (3,0.125) -- (2.875,0) -- (3.125,-0.125) -- (3,-0.25) -- (3,-4.875) -- (2.875,-5) -- (3.125,-5.125) -- (3,-5.25) -- (3,-6)-- (-5.25,-6) -- (-5.25,2) -- (3,2);

\begin{scope}[shift={(-0.5,0)}]

\def\int{1}

\node (v2) at (-3.5,1.5) {$\mathcal F_\int$};
\node (v3) at (-3.5,-4.5) {$\mathcal O_{\mathbb P^1_\int\times \mathbb P^1_\int}(-1,-1)$};
\draw   (-3.5,-6) edge (v3);
\draw  (v2) edge[<-] (-3.5,2);

\node (v6) at (-1.5,-0.5) {$\mathcal O_{p_\int}[1]$};

\draw[->]  (v6) edge node [fill=gray!20] {$\delta_\int^-$} (v2);

\draw[->]  (v6) edge node [fill=gray!20] {$\delta_\int^-$} (v2);

\draw[dotted, fill=white]  (-1.5,-5) ellipse (0.5 and 0.5);

\node (v1) at (-5,-1) {};

\node (v7) at (-3.5,-0.5) {$\mathcal O_{\mathbb P^1_\int\times \mathbb P^1_\int}(-1, -1)$};
\node (v5) at (-3.5,-2.5) {$\mathcal G_\int$};
\draw[->]  (v3) edge (v5);
\draw[->]  (v7) edge (v5);

\draw[->]  (v7) edge (v2);

\draw[->]  (v1) edge (v5);
\end{scope}

\begin{scope}[shift={(3.5,0)}]

\def\int{2}

\node (v2) at (-3.5,1.5) {$\mathcal F_\int$};
\node (v3) at (-3.5,-4.5) {$\mathcal O_{\mathbb P^1_\int\times \mathbb P^1_\int}(-1,-1)$};
\draw   (-3.5,-6) edge (v3);
\draw  (v2) edge[<-] (-3.5,2);

\node (v6) at (-1.3,-0.5) {$\mathcal O_{p_\int}[1]$};

\draw[->]  (v6) edge node [fill=gray!20] {$\delta_\int^-$} (v2);

\draw[->]  (v6) edge node [fill=gray!20] {$\delta_\int^+$} (v2);

\draw[dotted, fill=white]  (-1.5,-5) ellipse (0.5 and 0.5);

\node (v7) at (-3.5,-0.5) {$\mathcal O_{\mathbb P^1_\int\times \mathbb P^1_\int}(-1, -1)$};
\node (v5) at (-3.5,-2.5) {$\mathcal G_\int$};
\draw[->]  (v3) edge (v5);
\draw[->]  (v7) edge (v5);

\draw[->]  (v7) edge (v2);
\node (v9) at (0.5,-2) {};
\draw  (v6) edge (v9);
\end{scope}

\end{scope}

\draw[->] (-1.65,-7.35) -- node[fill=gray!20] {$\delta^+_1$}(-0.45,-8.8);

\end{tikzpicture} \caption{Resolution of the diagonal on $I_n\times I_n$.}
\label{fig:resIn}
\end{figure}
where we identify the left/right and top/bottom sides. 
\begin{proof}[Proof of \cref{thm:resIn}]
We now prove that this resolves the diagonal.
$I_n$ is covered by nodal conics $Y_i=I_n\setminus \left(\bigcup_{j\neq i, i+1} \PP^1_i\right)$. Therefore, $I_n\times I_n$ is covered by charts of the form $f_{ij}: Y_i\times Y_j\to I_n\times I_n$. We show that the pullback of our chain complex to $Y_i\times Y_j$ is quasi-isomorphic to the diagonal restricted to that chart.
\begin{itemize}
	\item When $i=j$, the pullback of our chain complex to  $Y_i\times Y_i$ is homotopic to the one given in \cref{lem:resolutionNodalConic}.
	\item On charts $Y_i\times Y_j$ where the distance between $i,j$ is greater than $1$, the pullback vanishes (as no sheaf in our resolution has support on these charts). These charts are disjoint from the diagonal, so we have agreement.
	\item On charts $Y_i\times Y_{i+1}$, the restriction of our resolution is a resolution of diagonal on the algebraic torus (which sits inside the $\PP_{i+1}^1\times \PP_{i+1}^1\subset Y_i\times Y_{i+1})$.
\end{itemize}
In conclusion, the diagonal of $D^b\Coh(I_n)$ has time one generation by weakly product bimodules. 
\end{proof}
\begin{rem} It is known that the diagonal dimension of a smooth elliptic curve is two by \cite{olander2023diagonal}. As there are degenerations of elliptic curves with central fiber $I_n$, one expects from the conjectured semi-continuity of diagonal dimension in \cite{elagin2021three} that $\Ddim(D^b\Coh(I_n)) = 2$ as well explaining why it was necessary to take a weak diagonal resolution.
\end{rem}

\printbibliography

\end{document}